\documentclass[oneside,11pt]{amsart}
\pdfoutput=1
\usepackage[pdftex, paperwidth=210mm, paperheight=11in]{geometry}
\usepackage{mathtools}
\usepackage[utf8]{inputenc}
\usepackage[T1]{fontenc}
\usepackage[USenglish]{babel}
\usepackage{hyphenat}
\usepackage{accents}

\usepackage{enumitem}
\setlist[enumerate]{left=3pt, topsep=0pt}
\setenumerate[1]{label=\upshape{(\arabic*)}}

\usepackage{graphicx}
\usepackage{xspace}
	\xspaceaddexceptions{" \, \: \nb \nobreakdash \text \textup ] \}}
\usepackage{url}
\usepackage[largesc]{newtxtext}
\usepackage[varg]{newtxmath}
\usepackage[scaled=1.05]{biolinum}
\renewcommand*\mathsf[1]{\textup{\textsf{#1}}}
\usepackage{mathbbol}
\usepackage{microtype}

\usepackage{xcolor}
\usepackage[pdfencoding=auto,psdextra]{hyperref}
\hypersetup{
    colorlinks,
    linkcolor={red!20!black},
    citecolor={blue!30!black},
    urlcolor={blue!40!black}
    }

\providecommand\texorpdfstring[2]{#1}

\makeatletter

\swapnumbers

\theoremstyle{plain}
    \newtheorem{theorem}[subsection]{Theorem}
    \newtheorem*{theorem*}{Theorem}
    \newtheorem{corollary}[subsection]{Corollary}
    \newtheorem*{corollary*}{Corollary}
    \newtheorem{proposition}[subsection]{Proposition}
    \newtheorem*{proposition*}{Proposition}
    \newtheorem{lemma}[subsection]{Lemma}
    \newtheorem*{lemma*}{Lemma}

    \newtheorem*{namedThm*}{\theoremname}
    \newcommand{\theoremname}{}
    \newenvironment{namedthm*}[1]
            {       \renewcommand{\theoremname}{#1}
                    \begin{namedThm*}
            }
            {       \end{namedThm*}}
    
\theoremstyle{definition}
    \newtheorem{definition}[subsection]{Definition}
    \newtheorem*{definition*}{Definition}
    \newtheorem{notation}[subsection]{Notation}
    \newtheorem*{notation*}{Notation}
    \newtheorem{convention}[subsection]{Convention}
    \newtheorem*{convention*}{Convention}
    
\theoremstyle{remark}
    \newtheorem{remark}[subsection]{Remark}
    \newtheorem*{remark*}{Remark}

\usepackage{needspace}

\newcommand*{\numberedparagraph}{\refstepcounter{subsection}\noindent\thesubsection{}.}

\newcommand*{\SideTitle}[1]{
	\bigbreak%
	\pdfbookmark[2]{\textbullet\quad#1}{\thesubsection{}}%
	\noindent{\small\textbullet}\enskip\:\textsc{#1}\smallbreak%
}

\newcommand*{\nb}{\nobreakdash}

\newcommand*{\dash}{\unskip\,\textemdash\,}

\newcommand*{\restr}{\mathord{\upharpoonright}}

\newcommand*{\medcup}{\raisebox{.30ex}{\scalebox{.90}{\ensuremath{\textstyle\bigcup}}}}
\newcommand*{\medcap}{\raisebox{.30ex}{\scalebox{.90}{\ensuremath{\textstyle\bigcap}}}}

\newcommand*{\Medcap}[2]{\EM{\medcap{_{#2}}{#1}_{#2}}}

\newcommand*{\Mediumleftrightarrow}{\mathrel{{\Leftarrow}\mkern-16mu{\Rightarrow}}}
\newcommand*{\SEquiv}{\Leftrightarrow}
\newcommand*{\LEquiv}{\mathrel{\;{\Mediumleftrightarrow}\;}}
\newcommand*{\SEquivDef}{\;\Leftrightarrow^{\textup{def}}\,}
\newcommand*{\Implies}{\mathrel{\;{=}\mkern-10mu{\Rightarrow}\;}}

\renewcommand*{\And}{\:\mathchar"3026\:}

\newcommand*{\subs}{\subseteq}
\newcommand*{\sups}{\supseteq}
\newcommand*{\X}{\times}

\newcommand*{\proves}{\mathrel{{|}\mkern-3mu{\relbar}}}

\newcommand*{\0}{\varnothing}

\newcommand*{\onto}{\twoheadrightarrow}

\newcommand*{\longsim}{\scalebox{1.3}{\ensuremath\sim}}
\newcommand*\isomap{\xrightarrow{\mkern-6mu\smash{\raisebox{-1ex}{\longsim}}\mkern-4mu}}%

\newcommand*{\In}{\,{\in}\,}

\newcommand*{\xbar}[1]{%
	\mkern1mu%
	\overline{\mkern-1mu{#1}\mkern-1mu}%
	\mkern1mu%
	{}
}

\newcommand{\EM}[1]{\bgroup\ensuremath{#1}\egroup\xspace}

\newcommand*{\@style}[2]{\EM{#1{#2}}}

\newcommand*{\cC}{\@style{\mathcal}C}
\newcommand*{\cD}{\@style{\mathcal}D}
\newcommand*{\cF}{\@style{\mathcal}F}
\newcommand*{\cI}{\@style{\mathcal}I}
\newcommand*{\cN}{\@style{\mathcal}N}
\newcommand*{\cM}{\@style{\mathcal}M}
\newcommand*{\cO}{\@style{\mathcal}O}
\newcommand*{\cP}{\@style{\mathcal}P}
\newcommand*{\cQ}{\@style{\mathcal}Q}
\newcommand*{\cS}{\@style{\mathcal}S}

\newcommand*{\bI}{\@style{\mathbb}I}
\newcommand*{\bL}{\@style{\mathbb}L}
\newcommand*{\bN}{\@style{\mathbb}N}
\newcommand*{\bV}{\@style{\mathbb}V}

\newcommand*{\sF}{\@style{\mathsf}F}

\newcommand*{\Ax}[1]{\textsc{#1}}

\newcommand*{\On}{\mathbb{O}\mathrm{n}}

\newcommand*{\Z}{\EM{\mathsf{Z}}}

\newcommand*{\ZFC}{\EM{\mathsf{ZFC}}}
\newcommand*{\ZFm}{\EM{\mathsf{ZF}^{-\!}}}

\newcommand*{\KPI}{\EM{\mathsf{KP}_{\infty}}}
\newcommand*{\KPIL}{\EM{\mathsf{KP}_{\infty} + (\mathbb{V}=\mathbb{L})}}

\newcommand*{\ZFmAleph}[1]{\EM{{\ZFm + \textup"\aleph_{#1}\textup{ exists"}}}}

\newcommand*{\ModelZFmAleph}[2]{\EM{#1 \models \ZFmAleph{#2}}}

\newcommand*{\omegaMod}{$\omega$-model\xspace}

\newcommand*{\wCK}{\EM{\omega_1^{\scriptscriptstyle\mathsf{CK}}}}

\DeclareMathOperator{\Dom}{\operatorname{Dom}}
\DeclareMathOperator{\Img}{\operatorname{Im}}
\newcommand{\Card}{\operatorname{Card}^{\ast\mkern-1mu}}
\newcommand{\cardinal}{cardinal\EM{^{\ast\mkern1mu}}}
\DeclareMathOperator{\Th}{\operatorname{Th}}
\DeclareMathOperator{\Cone}{\operatorname{Cone}}
\newcommand*{\unique}{\boldsymbol{\iota}\mkern.8mu}

\newcommand*{\ThL}[1]{\EM{\Th(\mathbb{L}_{#1}})}

\newcommand*{\Hull}[1]{\EM{\mathsf{H}^{\mathbb{L}_{#1}}\mkern-1mu}}
\newcommand*{\tHull}[1]{\EM{\xbar{\mathsf{H}}^{\mathbb{L}_{#1}}\mkern-1mu}}

\newcommand*{\seq}[1]{\bgroup{#1}^{< \omega}\egroup}

\newcommand*{\rec}{\leqslant_\mathsf{T}}

\newcommand*{\eqT}{\equiv_\mathsf{T}}

\newcommand*{\hypr}{\leqslant_\mathsf{h}}

\newcommand*{\eqH}{\mathrel{\equiv_\mathsf{h}}}

\newcommand*{\eq}[1]{\equiv_{#1}}

\newcommand*{\Triangle}{\vartriangleleft}

\newcommand*{\pWF}[1]{\EM{\operatorname{pseudo-WF}(#1)}}

\newcommand*{\All}[1]{\forall #1 \mkern1mu}
\newcommand*{\Exists}[1]{\exists #1 \mkern1mu}

\newcommand*{\Set}[2]{\EM{\{\mkern1mu {#1} \mid #2 \mkern1mu\}}}

\newcommand*{\Seq}[2]{\EM{(#1_#2)_{#2 < \omega}}}

\newcommand*{\Det}[1]{\EM{\operatorname{Det}(#1)}}
\newcommand*{\TDet}[1]{\EM{\operatorname{Turing-Det}(#1)}}
\newcommand*{\HTDet}[1]{\EM{\operatorname{Hyp-Turing-Det}(#1)}}
\newcommand*{\WTDet}[2]{\EM{\operatorname{Weak-Turing-Det}_{#1}(#2)}}

\newcommand*{\game}[1]{\EM{\textup{G}_\omega (#1)}}

\renewcommand*{\*}[1]{\bgroup\boldsymbol#1\egroup}%
\renewcommand*{\=}[1]{\bgroup\underline#1\egroup}%

\newcommand*{\up}[1]{\bgroup\textup{#1}\egroup}

\newcommand*{\surj}[1]{\mathrel{\gg^{\mkern-4mu{#1}}}}

\newcommand*{\pI}{Player~\textup{I}\xspace}
\newcommand*{\pII}{Player~\textup{II}\xspace}

\newcommand*{\SSigma}[2]{\EM{\Sigma^{#1}_{#2}}}
\newcommand*{\PPi}[2]{\EM{\Pi^{#1}_{#2}}}
\newcommand*{\DDelta}[2]{\EM{\Delta^{#1}_{#2}}}

\renewcommand*{\L}[1]{\EM{\mathbb{L}_{#1}}}
\newcommand*{\LL}[2]{\EM{\mathbb{L}_{#1}^{\!#2}}}
\newcommand*{\Lcard}[1]{\L{#1}\nb-\cardinal}

\newcommand*{\Aleph}[1]{\EM{\aleph_{#1}}}
\newcommand*{\AlephL}[2]{\EM{\aleph_{#1}^{\mathbb{L}_{#2}}}}

\newcommand*{\MM}[1]{\bgroup\ensuremath{\mathcal{M}_{#1}}\egroup}
\newcommand*{\inM}[1]{\mathrel{\in^\MM{#1}}}
\newcommand*{\TT}[1]{\EM{\mathsf{T}_{\mkern-2mu#1}}}
\renewcommand*{\SS}[1]{\EM{\mathcal{S}_{#1}}}

\renewcommand*{\phi}{\varphi}

\renewcommand*{\le}{\leqslant}

\renewcommand*{\ge}{\geqslant}

\newcommand*\Cpr{\looseness=-1}

\newcommand\ITEM[2][\smallbullet]{
    \begin{itemize}[leftmargin=1em, topsep=1pt, label=#1]
    	\item  {#2}
    \end{itemize}
}

\newcommand*\nudge[2]{\kern-#1 #2 \kern#1}

\begingroup
    \catcode`"=\active %
    \gdef"{\textup{\char34}}%
    \AtBeginDocument{\mathcode`"=\string"8000 }
\endgroup

\begingroup
    \global\mathchardef\normalcolon=\mathcode`: %

    \xdef\colon{\mathcode`:=\the\mathcode`: \colon}
    \catcode`:=\active %
    \global\let:=\colon
    \AtBeginDocument{\mathcode`:=\string"8000 }
\endgroup

\newcommand*\Maketitle{%
    \begingroup
        \def\uppercasenonmath##1{}
        \let\MakeUppercase\normalsize
        \global\let\@shorttitle\shorttitle
        \gdef\shorttitle{\small\@shorttitle}
        \global\let\@authors\authors
        \gdef\authors{\small\sc\@authors}
        \maketitle
    \endgroup
}
\makeatother

\begin{document}
\title				{Variations on $\Delta^1_1$ Determinacy and $\aleph_{\omega_1}$}
\thanks			{Presented at the 12\textsuperscript{th} Panhellenic Logic Symposium, Crete, June 2019. 
					To appear in the Journal of Symbolic Logic 2022\Cpr}
\author			{Ramez L. Sami}
\address			{Department of Mathematics. Université de Paris, 75205 Paris, Cedex 13, France.}
\email			{\Small\sf sami@univ-paris-diderot.fr}

\newgeometry{left=26mm, right=28mm}
\begin{abstract}
    \parindent=0pt
    We consider a seemingly weaker form of $\Delta^1_1$ Turing determinacy. 
    
    Let $2 \leqslant \rho < \omega_1^\textsc{ck}$, $\textrm{Weak-Turing-Det}_\rho (\Delta^1_1)$ is the statement:
    
        \quad \emph{Every $\Delta^1_1$ set of reals cofinal in the Turing degrees contains two Turing distinct, $\Delta^0_\rho$-equivalent reals.}
    
    We show in $\mathsf{ZF}^-$:

        \quad$\textrm{Weak-Turing-Det}_\rho (\Delta^1_1)$ implies: 
        for every $\nu < \omega_1^\textsc{ck}$ there is a transitive model: $M \models \mathsf{ZF}^- + \textup"\aleph_\nu \textup{ exists"}$.
    
    As a corollary:
    \begin{itemize}[leftmargin=1em]
        \item[] If every cofinal $\Delta^1_1$ set of Turing degrees contains both a degree and its jump, 
        then for every $\nu < \omega_1^\textsc{ck}$, there is a transitive model: $M \models \mathsf{ZF}^- + \textrm"\aleph_\nu \textrm{ exists"}$.
    \end{itemize}
    
    \begin{itemize}[leftmargin=1em, topsep=3pt, itemsep=1pt, label={\Large$\cdot$}]
        \item 	With a simple proof, this improves upon a well-known result of Harvey Friedman on the strength of Borel determinacy 
        		(though not assessed level-by-level).
        
        \item 	Invoking Tony Martin's proof of Borel determinacy, $\textrm{Weak-Turing-Det}_\rho (\Delta^1_1)$ implies $\Delta^1_1$ 			determinacy.
        
        \item 	We show further that, assuming $\Delta^1_1$ Turing determinacy, or Borel Turing determinacy, as needed:\\
            --\;	Every cofinal $\Sigma^1_1$ set of Turing degrees contains a ``hyp-Turing cone''\,: 
            			$\{\,x \in \mathcal{D} \mid d_0 \leqslant_\mathsf{T} x \leqslant_\mathsf{h} d_0 \}$.\\
            --\;	For a sequence $(A_k)_{k < \omega}$ of analytic sets of Turing degrees, each cofinal in $\mathcal{D}$, 
            			$\bigcap_k A_k$ is cofinal in $\mathcal{D}$.
    \end{itemize}
\end{abstract}
\restoregeometry

\vspace*{-0pt}
\pdfbookmark[1]{Variations on $\Delta^1_1$ Determinacy and $\aleph_{\omega_1}$}{}
\Maketitle

\section*{Introduction}
A most important result in the study of infinite games is Harvey Friedman's \cite{Fri71}, where it is shown that a proof of determinacy, for Borel games, would require \Aleph1 iterations of the power set operation \dash and this is precisely what Tony Martin used in his landmark proof~\cite{Mar75}. 

Our focus here is on the Turing determinacy results of \cite{Fri71}, concentrating instead on the theory \ZFm, rather than Zermelo's \Z. 
In the \DDelta11 realm, Friedman essentially shows that the determinacy of Turing closed \DDelta11 games \dash henceforth, \TDet{\DDelta11} \dash implies the consistency of the theories \ZFmAleph{\nu}, for all $\nu < \wCK$. 
He does produce a level-by-level analysis entailing, e.g., that the determinacy of Turing closed \SSigma{0}{n+6} games implies the consistency of \ZFmAleph{n}.%
    \footnote{\,Improved by Martin to \SSigma{0}{n+5}.}%
\textsuperscript{,}%
    \footnote{\,In \cite{M-S12} Montalbán and Shore greatly refine the analysis of the proof-theoretic strength of \Det{\Gamma}, 
    where $\PPi03 \subs \Gamma \subs \DDelta04$.}

Importantly, it was further observed by Friedman (unpublished) that these results extend to produce transitive models, rather than just consistency statements. 
See Martin's forthcoming~book~\cite{Mar00} for details, see also Van~Wesep's~\cite{Wes00}.
\smallbreak

We forgo in this paper the level-by-level analysis to provide in §\ref{sec:Transitive-models} a simple proof 
of the existence of transitive models of \ZFm with uncountable cardinals, from \TDet{\DDelta11}. 
In so doing, we show that the full force of Turing determinacy isn't needed. The main result is Theorem \ref{thm:Models-with-Alephs}, 
with a simply stated corollary. 
For context, by Martin's Lemma (see \ref{subsec:Turing-determinacy}), \TDet{\DDelta11} is equivalent to:
\ITEM{\emph{Every cofinal \DDelta11 set of Turing degrees contains a cone of degrees}, i.e., a set \Set{x \in \cD}{d_0 \rec x}.}

\begin{theorem*}[\ref{thm:Models-with-Alephs}]
    Let $2 \le \rho < \wCK$, and assume every \DDelta11 set of reals, cofinal in the Turing degrees, contains two Turing distinct, 
    \DDelta{0}{\rho}-equivalent reals. 
    For every $\nu < \wCK$, there is a transitive model\up: \ModelZFmAleph{M}{\nu}.
\end{theorem*}

\begin{corollary*}[\ref{cor:Models-with-Alephs}]
    If every cofinal \DDelta11 set of Turing degrees contains both a degree and its jump, then for every $\nu < \wCK$, 
    there is a transitive model\up: \ModelZFmAleph{M}{\nu}.
\end{corollary*}

In §\ref{sec:Properties-Sigma} several results are derived, showing that 
\TDet{\DDelta11} imparts weak determinacy properties to the class \SSigma11, such as [\ref{thm:Hyp-Turing-Det-Sigma}]\,:
    \ITEM{\emph{Every cofinal \SSigma11 set of degrees includes a set \Set{x \in \cD}{d_0 \rec x \And x \hypr d_0}, for some $d_0 \in \cD$}.}
Or, from Borel Turing determinacy, [\ref{thm:Intersection of Sigma cofinal}]\,:
    \ITEM{\emph{If \Seq{A}{k} is a sequence of cofinal analytic subsets of \cD, then \Medcap{A}{k} is cofinal in \cD.}}
\medbreak

I wish to thank Tony Martin for inspiring exchanges on the present results. 
He provided the argument for Remark \ref{rem:Too-Weak}, below, and observed that my first proof of Theorem \ref{thm:Hyp-Det-Sigma} was needlessly complex. 
Parts of §\ref{sec:Properties-Sigma} go back to the author's dissertation~\cite{Sam76}, it is a pleasure to acknowledge Robert~Solovay's direction. 

\section{Preliminaries and notation}

The effective descriptive set theory we shall need, as well as basic hyperarithmetic theory, is from Moschovakis' \cite{Mos09}, 
whose terminology and notation we follow.
For the theory of admissible sets, we refer to Barwise's~\cite{Bar75}.
Standard facts about the \bL\nb-hierarchy are used without explicit mention: see Devlin's \cite{Dev77}, or Van~Wesep's~\cite{Wes00}.
\smallbreak

$\cN = \omega^\omega = \bN^\bN$ denotes Baire's space (the set of \emph{reals}), and \cD the set of Turing degrees. 
Subsets of \cD shall be identified with the corresponding (Turing closed) sets of reals.
$\rec$, $\hypr$, and $\eqT$, $\eqH$ shall denote, respectively, Turing and hyperarithmetic (i.e. \DDelta11) reducibility, and equivalence.%

\subsection{The ambient theories.}

Our base theory is \ZFm, \Ax{Zermelo-Fraenkel} set theory stripped of the Power Set axiom.%
    \footnote{\,All implicit uses of Choice herein are \ZFm-provable.}
\cN or \cD may be proper classes in this context, yet speaking of their `subsets' (\DDelta11, \SSigma11, Borel or analytic) can be handled as usual, as these sets are codable by integers, or reals.
Amenities such as \Aleph1 or \L{\omega_1} aren't available but, since our results here are global (i.e.,~\DDelta11) rather than local, the reader may use instead the more comfortable $\ZFm + "\cP^2(\omega) \up{ exists}"$.

\KPI is the theory \Ax{Kripke--Platek + Infinity}. Much of the argumentation below involves \omegaMod{s} of \KPI \dash familiarity with their properties is assumed.

\subsection{Turing determinacy.}
\label{subsec:Turing-determinacy}

A set of reals $A \subs \cN$ is said to be \emph{Turing-cofinal} if, for every $x \in \cN$, there is $y \in A$, such that $x \rec y$. 
A \emph{Turing cone} is a set $\Cone(c) = \Set{x \in \cN}{c \rec x}$, where $c \in \cN$.
For a class of sets of reals $\Gamma$, \Det{\Gamma} is the statement that infinite games \game{A} where $A \in \Gamma$ are determined, whereas \TDet{\Gamma} stands for the determinacy of games \game{A} restricted to Turing closed sets $A \in \Gamma$.
Recall the following easy, yet central:

\begin{namedthm*}{Martin's Lemma \normalfont{\cite{Mar68}}}
For a Turing closed set $A \subs \cN$, the infinite game \game{A} is determined \emph{iff} $A$~or its complement contains a Turing cone. \qed
\end{namedthm*}

\subsection{Constructibility and condensation.}
\label{subsec:Constructibility}

For an ordinal $\lambda > 0$, and $X \subs \L{\lambda}$, $\Hull{\lambda}(X)$ denotes the set of elements of \L{\lambda} definable from parameters in $X$, and $\tHull{\lambda}(X)$ its transitive collapse.
For $X = \0$, one simply writes \Hull{\lambda} and \tHull{\lambda}. Gödel's Condensation Lemma is the relevant tool here.
Note that, since $\L{\lambda} = \tHull{\lambda}(\lambda) = \Hull{\lambda}(\lambda)$, all elements of \L{\lambda} are definable in \L{\lambda} from ordinal parameters.

\subsection{Reflection.}

The following reflection principle will be used a few times, to make for shorter proofs.%
    \footnote{\,Longer ones can always be produced using \DDelta11 selection $+$ \SSigma11 separation.}
A property $\Phi(X)$ of subsets $X \subs \cN$ is said to be "\PPi11 \emph{on} \SSigma11" if, for any \SSigma11 relation $U \subs \cN \X \cN$, the set \Set{x \in \cN}{\Phi(U_x)} is \PPi11. 

A simple example: let $A \subs \cN$ be \SSigma11, and set: $\Theta(X) \SEquiv X \cap A = \0$. $\Theta(X)$ is a \PPi11 on \SSigma11 property. 

\begin{theorem*}
    Let $\Phi(X)$ be a \PPi11 on \SSigma11 property. 
    For any \SSigma11 set $S \subs \cN$ such that\/ $\Phi(S)$ there is a \DDelta11 set $D \sups S$ such that $\Phi(D)$.
\end{theorem*}

\begin{proof}
See Kechris'~\cite[§35]{Kec95} for a boldface version, easily transcribed to lightface.
\end{proof}

\section{Weak-Turing-Determinacy}
\label{sec:Weak-Turing-determinacy}

Examining what's needed to derive the existence of transitive models from Turing determinacy hypotheses, 
it is possible to isolate a seemingly weaker statement.
For $1 \le \rho < \wCK$, let $x \eq{\rho} y$ denote \DDelta{0}{\rho}-equivalence on \cN, that is: 
$x \in \DDelta{0}{\rho}(y) \mathrel{\&} y \in \DDelta{0}{\rho}(x)$. 
$\eq{1}$ is just Turing equivalence.

\begin{definition}
    For a class $\Gamma$, and $2 \le \rho < \wCK$, define \WTDet{\rho}{\Gamma}:
	\ITEM[]{\emph{Every Turing-cofinal set of reals $A \in \Gamma$ has two Turing distinct elements $x, y \in A$ 
    			such that~$x \eq{\rho} y$.}}
\end{definition}

For any recursive $\rho \ge 2$, \WTDet{\rho}{\DDelta11} will suffice to derive the existence of transitive models of \ZFm with uncountable cardinals (clearly, larger values for $\rho$ yield formally weaker sentences).
The property lifts from \DDelta11 to \SSigma11 \dash note that it is, \emph{a~priori}, asymmetric.

\begin{theorem}
\label{thm:WTDelta->WTSigma}
    Let $2 \le \rho < \wCK$, $\WTDet{\rho}{\DDelta11} \Implies \WTDet{\rho}{\SSigma11}$.
\end{theorem}

\begin{proof}
Assume \WTDet{\rho}{\DDelta11}. Let $S \in \SSigma11$ and suppose there are no Turing distinct $x, y \in S$ such that $x \eq{\rho} y$, that is
\[
        \All{x,y} (x, y \in S \And x \eq{\rho} y \Implies x \eqT y).
\]
This is a statement $\Phi(S)$, where $\Phi(X)$ is easily checked to be a \PPi11 on \SSigma11 property.
Reflection yields a \DDelta11 set $D \sups S$ such that $\Phi(D)$. By \WTDet{\rho}{\DDelta11}, $D$ is not Turing-cofinal; \emph{a~fortiori}, $S$~isn't either.
\end{proof}

\begin{remark}
\label{rem:Too-Weak}
One may be tempted to substitute for \WTDet{\rho}{\DDelta11} a simpler hypothesis:
\ITEM[]{\emph{Every Turing-cofinal $\smash{\DDelta11}$ set of reals has Turing distinct elements $x, y$, such that $x \eqH y$}.}
It turns out to be too weak and, indeed, provable in Analysis. 
(Tony Martin, private communication: building on his paper \cite{Mar76}, he shows that every uncountable \DDelta11 set of reals contains two Turing distinct reals, in every hyperdegree $\ge$ Kleene's \cO.) 

The simpler, weaker, condition does suffice however when asserted about the class \SSigma11, see Theorem~\ref{thm:Models-with-Alephs-from-Sigma}, below.
\end{remark}

\section{Transitive models from Weak-Turing-Determinacy}
\label{sec:Transitive-models}

We now state the main result, and a simple special case. The proof is postponed toward the end of the present section.

\begin{theorem}
\label{thm:Models-with-Alephs}
    Let $2 \le \rho < \wCK$, and assume \WTDet{\rho}{\DDelta11}. For every $\nu < \wCK$, there is a transitive model\up: \ModelZFmAleph{M}{\nu}. 
\end{theorem}

\begin{corollary}
\label{cor:Models-with-Alephs}
    If every cofinal \DDelta11 set of Turing degrees contains both a degree and its jump, then for every $\nu < \wCK$, there is a transitive model\up: \ModelZFmAleph{M}{\nu}.\qed
\end{corollary}

\SideTitle{Term models.}

Given a complete theory%
    \footnote{\,Complete theories are meant here to be \emph{consistent}, and \emph{deductively closed}.}
$U \sups \KPIL$, one constructs its term model.
To be specific: owing to the presence of the axiom $\bV = \bL$, to every formula $\psi(v)$ is associated $\xbar\psi(v)$ such that $U \proves \Exists{v}\psi(v) \SEquiv \Exists{!v}\xbar\psi(v)$, just take for $\xbar\psi(v)$ the formula:\, $\psi(v) \land (\All{w <_\bL v}) \neg \psi(w)$.

Let now $(\phi_n(v))_{n < \omega}$ be a recursive in $U$ enumeration of the formulas $\phi(v)$, in the single free variable $v$, 
having $U \proves \Exists{!v} \phi(v)$. 
Using, as metalinguistic device, $(\unique v) \phi(v)$ for "\emph{the unique $v$ such that $\phi(v)$}" set:
\vspace*{-1ex}
\begin{align*}
        	M_U 	\nudge{6pt}{&}  	=   			\Set{n \in \omega}{\All{\ell < n,}\; U \proves (\unique v)\phi_n \not= (\unique v)\phi_\ell},
\intertext{and define on $M_U$ the relation $\in_U$:}
        m \in_U  n 			&	 	\LEquiv		U \proves (\unique v)\phi_m \in (\unique v)\phi_n .
\end{align*}
$(M_U, \in_U)$ is a prime model of $U$ and, $U$ being complete, $(M_U, {\in_U}) \rec U$. 
Using the canonical $1{\up-}1$ enumeration $\omega \to M_U$, substitute $\omega$ for $M_U$ and remap $\in_U$ accordingly. 
The resulting model $\MM{U} = (\omega, \inM{U})$ shall be called the \emph{term model} of $U$. 
The function $U \mapsto \MM{U}$ is recursive.

Whenever \MM{U} is an \omegaMod, we say that $a \subs \omega$ is realized in \MM{U} 
if there is $\mathring{a} \in \omega$ such that $a = \Set{k \in \omega}{\mathbf{k}^\MM{U} \inM{U} \mathring{a}}$. 
We state, for later reference, a couple of standard facts.

\begin{proposition}
\label{prop:Realized-reals}
    Let $U$ be as above. If \MM{U} is an \omegaMod, and $a \subs \omega$ is realized in \MM{U}, then\up:
    \begin{enumerate}
        \item	For all $x \hypr a$, $x$ is realized in \MM{U}.
        \item	$a \rec U$. 
        \item	Thus $U$ is not realized in \MM{U}, lest its Turing jump $U'$ be realized in \MM{U}, causing $U' \rec U$.\qed
    \end{enumerate}
\end{proposition}

Note that if $U = \ThL{\alpha}$, where $\alpha$ is admissible, then \MM{U} is a copy of \Hull{\alpha}.
Hence $\MM{U} \cong \L{\beta}$, for some $\beta \le \alpha$. The following easy proposition is quite familiar.

\begin{proposition}
\label{prop:Cofinal-term-models}
    Assume $\bV = \bL$. For cofinally many countable admissible $\alpha$'s, $\L{\alpha} = \Hull{\alpha}$, 
    equivalently: $\MM{\ThL{\alpha}} \cong \L{\alpha}$.
\end{proposition}

\begin{proof}
Suppose not. 
Let $\lambda$ be the sup of the admissible $\alpha$'s having $\L{\alpha} = \Hull{\alpha}$, and let $\kappa > \lambda$ be the first admissible such that $\lambda$ is countable in \L{\kappa}. 
Since $\lambda$ is definable and countable in \L{\kappa}, $\lambda \cup \{\lambda\} \subs \Hull{\kappa}$.
It follows readily that $\L{\kappa} = \tHull{\kappa} = \Hull{\kappa}$, a contradiction.
\end{proof}

\SideTitle{Cardinality in the constructible levels.}

Set theory within the confines of \L{\lambda}, $\lambda$ an arbitrary limit ordinal, imposes some contortions.
For technical convenience, the notion of cardinal needs to be slightly twisted \dash for a time only.

\begin{definition}
\label{def:cardinality}
    \begin{enumerate}
        \item	For an ordinal $\alpha$, $\Card(\alpha) = \min_{\xi \le \alpha}(\textit{there is a surjection}\ \xi \to \alpha)$.
        \item	$\alpha$ is a \cardinal if $\alpha = \Card(\alpha)$.
        \item	$\Card_\lambda \subs \L{\lambda}$ is the class of infinite \cardinal{'s} as computed in \L{\lambda}.
    \end{enumerate}
\end{definition}

\numberedparagraph{}%
\label{subsec:Bijection}
Note that, for $\lambda$ limit, from a surjection $g : \gamma \to \alpha$ in \L{\lambda}, one can extract $a \subs \gamma$ and ${\Triangle} \subs a \X a$ such that ${g \restr a} : (a, \Triangle) \isomap (\alpha, \in)$,
and both $(a, \Triangle)$, $g \restr a$ are in \L{\lambda}. 
Further, if $\lambda$ is admissible, in \L{\lambda} the altered notion of cardinality and the standard one coincide.

\begin{convention}
\label{label:Convention}
    For simplicity's sake, the assertion "\emph{\Aleph{\nu} exists in \L{\lambda}}" should be understood as: 

    \emph{There is an isomorphism $\nu+1 \isomap J$, where $J$ is an initial segment of\/ $\Card_\lambda$}.\\
    Note that its negation is equivalent in \KPI to: \emph{There is $\kappa \le \nu$ such that\/ $\Card_\lambda \cong \kappa$}.
    The notation \AlephL{\nu}{\lambda} carries here the obvious meaning.
\end{convention}

We shall need the following result, readily proved using the Jensen fine structure techniques of~\cite{Jen71}. A~direct proof is provided in the Appendix.

\begin{proposition}
\label{prop:Folklore}
    For $\lambda$ a limit ordinal, if\/ $\L{\lambda} \models "\*\mu > \omega \up{ is a successor \cardinal}"$ then $\L{\mu} \models \ZFm$.
\end{proposition}

\SideTitle{\texorpdfstring{The theories \TT{\nu}}{The theories $T_\nu$}.}

Let \cM be an \omegaMod of \KPI. The wellfounded part of  $\On^\cM$ `includes' \wCK. 
For $\nu < \wCK$, pick~$e_\nu$ a recursive index for a wellordering $<_{e_\nu}$ of a subset of $\omega$, of length $\nu$. 
Using $e_{\nu}$, statements about~$\nu$ can tentatively be expressed in \KPI. 
In \cM, the truth of such statements is independent of the choice of $e_{\nu}$.
Indeed, $<_{e_{\nu}}$ is realized in \cM, and its realization is isomorphic in \cM to the \cM\nb-ordinal of order-type $\nu$, to be denoted $\nu^\cM$.
For a formula $\phi(v, \cdots)$, we write $\cM \models \phi(\=\nu, \cdots)$, instead of a `translated'  $\cM \models \xbar{\phi}(\*{e_\nu}, \cdots)$.

\begin{definition}
    For $\nu < \wCK$, \TT{\nu} is the theory
    \[
        \KPIL + "\up{for all limit $\lambda$, \Aleph{\=\nu+1} doesn't exist in \L{\lambda}}".
    \]
\end{definition}

This definition is clearly lacking: a recursive index $e_{\nu}$ coding the ordinal $\nu$ is not made explicit. 
This is immaterial, as we shall be interested only in \omegaMod{s} of \TT{\nu}. They possess the following rigidity property.

\begin{lemma}
\label{lem:Rigidity-ppty}
    Let $\nu < \wCK$, and \MM1, \MM2 be \omegaMod{s} of\/ \TT{\nu}.
    Let $u \in \On^\MM1$, and $w, w_\ast \in \On^\MM2$, for any two isomorphisms $f : \LL{u}{\MM1} \isomap \LL{w}{\MM2}$ and $f_\ast : \LL{u}{\MM1} \isomap \LL{w_\ast}{\MM2}$, $f = f_\ast$.
\end{lemma}

\begin{proof}
By an easy reduction, it suffices to prove this for $u$, a limit \MM1-ordinal.

Let $<_{1}$ denote the ordering of $\On^\MM1$ in \MM1, and set $\cC_u = \Set{c <_{1}  u}{\MM1 \models \*c \in \Card_\*u}$.
The relevant claim here is that $(\cC_u, <_{1})$ is wellordered. Indeed, since $\MM1 \models \TT{\=\nu}$,
\[
            \MM1 \models "\Aleph{\=\nu+1}\up{ doesn't exist in }\L{\*u}".
\]
Hence, as observed in \ref{label:Convention}, there is $o \in \On^\MM1$ with 
\[
            \MM1 \models \*o \le \=\nu+1  \And  \Card_\*u  \cong  \*o.
\] 
The isomorphism in \MM1 induces an actual $<_{1}$\nb-isomorphism: $\cC_u \cong \Set{x}{x <_{1} o}$.
Since \MM1 is an \omegaMod, $\nu^\MM1$ and $o$ are in its wellfounded part, thus the claim.

First, one checks that $f$ and $f_\ast$ agree on the \MM1-ordinals $o <_{1} u$, using induction on $\cC_u$. 
Clearly, for $o \le_{1} \omega^\MM1$, $f(o) = f_\ast(o)$.
Set $\kappa_u(o) = \Card(o)$, \emph{as evaluated in} \LL{u}{\MM1}, and show by induction on $c \in \cC_u$:
\[
       \up{for all}\; o <_{1}  u, \enskip\;   \kappa_u(o) = c \Implies f(o) = f_\ast(o).
\]

Assume the inductive hypothesis for $c' <_{1} c$. 
Whenever $o <_{1} c$, $\kappa_u(o) < c$, hence $f(o) = f_\ast(o)$. It follows easily that $f(c) = f_\ast(c)$. 
Let now $o <_{1} u$ have $\kappa_u(o) = c$. 
Inside \LL{u}{\MM1}, $(o, \in)$ is isomorphic to an ordering `$s = (a, \Triangle)$', 
where $a \subs c$ and ${\Triangle} \subs c \X c$, (see \ref{subsec:Bijection}). 
Since $f$ and $f_\ast$ agree on the \MM1-ordinals up to~$c$, one readily checks $f(s) = f_\ast(s)$. 
In \MM2 now, the common value $f(s)$ is isomorphic to both the ordinals $f(o)$ and $f_\ast(o)$, hence $f(o) = f_\ast(o)$.

This entails $w = w_\ast$ and $\LL{w}{\MM2} = \LL{w_\ast}{\MM2}$.
Now, any $x \in \LL{u}{\MM1}$ is definable in \LL{u}{\MM1} from \MM1-ordinals (see~\ref{subsec:Constructibility}), 
thus $f(x)$ and $f_\ast(x)$ satisfy in \LL{w}{\MM2} the same definition from equal parameters, hence $f(x) = f_\ast(x)$.
\end{proof}

\SideTitle{Pseudo-wellfounded models.}

A relation ${\Triangle} \subs \omega \X \omega$ is said to be \emph{pseudo-wellfounded} if every nonempty $\DDelta11(\Triangle)$ subset of $\omega$ has a $\Triangle$-minimal element. 
By the standard computation, this is a \SSigma11 property.%
    \footnote{\,We shall use, in complexity computations, the classic result of Kleene:
    	\emph{Given a \SSigma11 predicate $\cS(x, y, -)$, the predicate $(\All{y \hypr x})\cS(x, y, -)$ is \SSigma11, and dually for \PPi11}.
    	See \cite[§4D.3]{Mos09} for a more general result.}
Indeed, we may define it, for $\*E \subs \omega \X \omega$, as:
\[
          \pWF{\*E} 	\SEquivDef 	(\All{X \hypr \*E})\bigl(X \neq \0 \Implies (\Exists{k \In X})(\All{m \In X})\neg(m \,{\*E}\, k)\bigr).
\]

\begin{definition}
    For $\nu < \wCK$, \SS{\nu} is the set of theories:
    \[
            \SS{\nu} = \Set{U}{U \up{ is a complete extension of \TT{\nu}, and \MM{U} is pseudo-wellfounded}}.
    \]
\end{definition}

Easily, \SS{\nu} is \SSigma11. Indeed, the first clause in its definition is arithmetical, 
while the second reads "\pWF{\inM{U}}", where the function $U \mapsto {\inM{U}}$ is recursive.

Note further: for $U \in \SS{\nu}$, \MM{U} is an \omegaMod. 
The sets \SS{\nu} play a central role in the proof. They are sparse, in the following sense. 

\begin{proposition}
    For $\nu < \wCK$, no two distinct members of \SS{\nu} have the same hyperdegree.
\end{proposition}

\begin{proof}
Let $U_1, U_2 \in \SS{\nu}$ have $U_1 \eqH U_2$, and let \MM1, \MM2 stand for \MM{U_1}, \MM{U_2}. 
We'll obtain $U_1 = U_2$ by showing $\MM1 \cong \MM2$. Define a relation between `ordinals' $u \in \MM1$ and $w \in \MM2$\,:
\begin{align*}
        	u \simeq w 			& 		\LEquiv  	\Exists{f} ( f : \LL{u}{\MM1} \isomap \LL{w}{\MM2} ).
\intertext{%
        Set $\bI_1 = \Dom(\simeq)$, and $\bI_2 = \Img(\simeq)$. $\bI_1$ and $\bI_2$ are initial segments of $\On^\MM1$ and $\On^\MM2$, respectively. 
        Using Lemma \ref{lem:Rigidity-ppty}, the relation "$u \simeq w$" defines a bijection $\bI_1 \to \bI_2$ which is, indeed, the restriction of an isomorphism:
}%
        \sF : \medcup_{u \in \bI_1}   \LL{u}{\MM1} 
        				\nudge{4pt}{&}	\isomap 	\medcup_{w \in \bI_2} \LL{w}{\MM2}.
\intertext{Note that, by the same lemma,}
		u \simeq w			& 		\LEquiv 	\Exists{!f}( f : \LL{u}{\MM1} \isomap \LL{w}{\MM2} ).
\intertext{The RHS here reads: $\Exists{!f} \cI(f, U_1, u, U_2, w)$, where \cI is a \DDelta11 predicate, hence:}
        	u \simeq w			& 		\LEquiv 	\Exists{f \hypr U_1 \oplus U_2}(f : \LL{u}{\MM1} \isomap \LL{w}{\MM2}).
\end{align*}
By the standard computation, the relation "$u \simeq w$" is $\DDelta11(U_1 \oplus U_2)$ $[\, = \DDelta11(U_1) = \DDelta11(U_2)]$. 
Consequently, $\bI_1$ and $\bI_2$ are also $\DDelta11(U_1)$ $[\, = \DDelta11(U_2)]$. 
$\MM1,\,  \MM2$ being pseudo-wellfounded, $\On^\MM1 - \bI_1$ and $\On^\MM2 - \bI_2$ each, if nonempty, has a minimum. 
Denote $m_1, m_2$ the respective potential minima, and consider the cases:
\begin{itemize}[left=3pt, topsep=2pt, itemsep=2pt, label={--}]
    \item	$\On^\MM1 - \bI_1$ and $\On^\MM2 - \bI_2$ are both nonempty. 
                This isn't possible, as $\sF$ would be the isomorphism $\sF : \LL{m_1}{\MM1} \isomap \LL{m_2}{\MM2}$, 
                entailing $m_1 \in \bI_1$ and $m_2 \in \bI_2$.

    \item 	$\bI_1 = \On^\MM1$ and $\On^\MM2 - \bI_2 \ne \0$. 
                Here $\MM1 = \medcup_{u \in \bI_1} \LL{u}{\MM1}$, and  thus $\sF : \MM1 \isomap \LL{m_2}{\MM2}$. 
                $U_1$ is now the theory of \LL{m_2}{\MM2}, hence it is realized in \MM2. 
                Since $U_2 \eqH U_1$, by Prop.~\ref{prop:Realized-reals}(1), $U_2$ is also realized in \MM2 (that's \MM{U_2}). 
                This contradicts (3) of the same proposition.

    \item 	The third case, symmetric of the previous one, is equally impossible.

    \item 	The remaining case: $\bI_1 = \On^\MM1$ and $\bI_2 = \On^\MM2$. 
               Here $\MM1 = \medcup_{u \in \bI_1} \LL{u}{\MM1}$ and $\MM2 = \medcup_{w \in \bI_2} \LL{w}{\MM2}$,
               thus $\sF : \MM1 \isomap \MM2$ is the desired isomorphism.\qedhere
\end{itemize}
\end{proof}

\begin{proof}[\textbf{Proof of Theorem \ref{thm:Models-with-Alephs}}]
Our hypothesis is \WTDet{\rho}{\DDelta11}, and we may work entirely in \bL. 

Fix any $\nu < \wCK$, towards a transitive model of \ZFmAleph{\nu}.

\textsc{Claim.} There is a limit ordinal $\lambda$, such that: \Aleph{\nu+1} exists in \L{\lambda}.

Suppose no such $\lambda$ exists. 
It follows that for all admissible $\alpha > \omega$, $\L{\alpha} \models \TT{\nu}$. This entails that \SS{\nu} is Turing-cofinal: 
indeed, since $\bV = \bL$, using Prop.~\ref{prop:Cofinal-term-models}, given $x \subs \omega$ there is an $\alpha > \omega$, admissible, 
such that $x \in \L{\alpha}$ and $\MM{\ThL{\alpha}} \cong \L{\alpha}$. 
Thus $x \rec \ThL{\alpha}$ and, \MM{\Th(\L{\alpha})} being wellfounded, $\ThL{\alpha} \in \SS{\nu}$. 

Invoking now \WTDet{\rho}{\DDelta11} and Theorem~\ref{thm:WTDelta->WTSigma}, \WTDet{\rho}{\SSigma11} holds. 
Hence, there are distinct $U_1, U_2 \in \SS{\nu}$ such that $U_1 \eq{\rho} U_2$, contradicting the previous proposition.
\qed\textsubscript{\,\textsc{Claim}}

Let now $\lambda$ be as claimed, and set $\mu = \AlephL{\nu+1}{\lambda}$. 
In \L{\lambda}, $\mu$ is a successor \cardinal hence, by Prop.~\ref{prop:Folklore}, $\L{\mu} \models \ZFm$. 
Further, for $\xi \le \nu$,  $\AlephL{\xi}{\lambda} < \mu$ and \AlephL{\xi}{\lambda} is an \L{\mu}\nb-cardinal (now in the usual sense), 
hence $\L{\mu} \models \ZFmAleph{\nu}$.
\end{proof}

Note the following byproduct of the previous proposition, and the proof just given 
(substituting $U_1 \eqH U_2$ for $U_1 \eq{\rho} U_2$, in the proof) 
\dash in contradistinction to Remark~\ref{rem:Too-Weak}.

\begin{theorem}
\label{thm:Models-with-Alephs-from-Sigma}
    Assume every Turing-cofinal \SSigma11 set of reals has two Turing distinct elements $x, y$, such that $x \eqH y$. 
    For every $\nu < \wCK$, there is a transitive model\up: $M \models \ZFmAleph{\nu}$.\qed
\end{theorem}

An easy consequence of the main result: \WTDet{\rho}{\DDelta11} implies full \DDelta11 determinacy. 
The proof proceeds via Martin's Borel determinacy theorem: 
no direct argument is known for this sort of implication \dash apparently first observed by Friedman for \TDet{\DDelta11}.

\begin{theorem}
\label{thm:Full-Det}
    For $2 \le \rho < \wCK$, \WTDet{\rho}{\DDelta11} implies \Det{\DDelta11}.
\end{theorem}

\begin{proof}
Assume \WTDet{\rho}{\DDelta11}. 
Let $A \subs \cN$ be \DDelta11, say $A \in \SSigma{0}{\nu}$ where $\nu < \wCK$. 
Applying Theorem \ref{thm:Models-with-Alephs}, there is a transitive \ModelZFmAleph{M}{\nu}. 
Invoking (non-optimally) Martin's main result from \cite{Mar76} inside $M$, \SSigma{0}{\nu} games are determined. 
The statement "\emph{the game \game{A} is determined}" is \SSigma12. 
By Mostowki's absoluteness theorem, being true in $M$, it holds in the universe: \game{A} is indeed determined.
\end{proof}

\section{$\Delta^1_1$ determinacy and properties of $\Sigma^1_1$ sets}
\label{sec:Properties-Sigma}

We proceed now to show that \DDelta11 determinacy imparts weak determinacy properties to the class \SSigma11. 
In view of Theorem \ref{thm:Full-Det}, there is no point, here, in working from weaker hypotheses.

\begin{definition}
    The hyp-Turing cone with vertex $d \in \cD$ is the set of degrees 
    \[
        \Cone_\mathsf{h}(d) = \Cone(d) \cap \DDelta11(d) = \Set{x \in \cD}{d \rec x \And x \hypr d}.
    \]
    \HTDet{\Gamma} is the statement: \emph{Every cofinal set of degrees $A \in \Gamma$ contains a hyp-Turing~cone}.
\end{definition}

\begin{theorem}
\label{thm:Sigma-seq-degrees}
    Assume \TDet{\DDelta11}. If \Seq{S}{k} is a \SSigma11 sequence of Turing-cofinal sets of degrees, 
    then $\Medcap{S}{k} \neq \0$ \dash and, indeed, \Medcap{S}{k} contains a hyp-Turing cone.
\end{theorem}

\begin{proof}
Let the $S_{k}$'s be given as the sections of a \SSigma11 relation $S \subs \omega \X \cN$, and assume \Medcap{S}{k} contains no hyp-Turing cone: 
$\All{x \in \cN}(\Cone_\mathsf{h}(x) \not\subs \Medcap{S}{k})$, i.e.,%
\[
        \All{x \in \cN} \Exists{y \hypr x} (x \rec y \And y \notin \Medcap{S}{k}).
\]
This is a statement $\Phi(S)$, where $\Phi(X)$ is a \PPi11 on \SSigma11 property of subsets $X \subs \omega \X \cN$.
Reflection yields a \DDelta11 relation $D \sups S$ such that $\Phi(D)$.
Shrink $D$, if need be, to ensure that its sections $D_{k}$ are Turing closed, preserving $\Phi(D)$ and $D \sups S$.
Now, $D_{k} \sups S_{k}$ and \Medcap{D}{k} contains no hyp-Turing cone.
A contradiction ensues using \TDet{\DDelta11} $\,+\,$ Martin's Lemma: each $D_{k}$, being cofinal in~\cD, contains a Turing cone hence, easily, so does \Medcap{D}{k}.
\end{proof}

The converse is immediate. Indeed, if \TDet{\DDelta11} fails, by Martin's Lemma there is a \DDelta11 set $A \subs \cD$, 
such that both $A$ and ${\sim}A$ are cofinal in \cD, and the \DDelta11 `sequence' $\langle A, {\sim}A \rangle$ has empty intersection.
Relativizing \ref{thm:Sigma-seq-degrees}, one readily gets:

\begin{corollary}
\label{thm:Intersection of Sigma cofinal}
    Assume Borel Turing determinacy. If \Seq{A}{k} is a sequence of cofinal analytic sets of Turing degrees, then \Medcap{A}{k} is cofinal in \cD.\qed
\end{corollary}

An interesting special case of \ref{thm:Sigma-seq-degrees}, where the `sequence' $(S_{k})_{k<1}$ is a single \SSigma11 term.

\begin{theorem}
\label{thm:Hyp-Turing-Det-Sigma}
    \TDet{\DDelta11} implies \HTDet{\SSigma11}.\qed
\end{theorem}

In view of Theorem \ref{thm:Full-Det}, the implication is an equivalence. A similar result obtains for full determinacy, as well:

\begin{definition}
    For a game \game{A}, a strategy $\sigma$ for \pI is called a hyp-winning strategy if \,$\All{\tau \hypr \sigma}(\sigma{\ast}\tau \in A)$,
    i.e., applying $\sigma$, \pI wins against any $\DDelta11(\sigma)$ sequence of moves by \pII.
\end{definition}

\begin{theorem}
\label{thm:Hyp-Det-Sigma}
    Assume \Det{\DDelta11}. For $S \in \SSigma11$, one of the following holds for \game{S},
    \begin{enumerate}
        \item	\pI has a hyp-winning strategy.
        \item	\pII has a winning strategy.
    \end{enumerate}
\end{theorem}

\begin{proof}
Say $S \in \SSigma11$, and \pI has no hyp-winning strategy for \game{S}:
$\All{\sigma} \Exists{\tau \hypr \sigma}(\sigma\ast\tau \notin S)$.
Much as in the proof of \ref{thm:Sigma-seq-degrees}, Reflection yields a \DDelta11 set $D \sups S$ such that 
\pI has no hyp-winning strategy for \game{D}, hence no winning strategy.
Invoking \Det{\DDelta11}, \pII has a winning strategy for \game{D} which is, \emph{a~fortiori}, winning for \game{S}.
\end{proof}

\section{Appendix}

The point of the present section is to sketch a proof of Proposition~\ref{prop:Folklore}, without dissecting the \bL~construction 
\dash albeit with a recourse to admissible sets. Finer results most certainly hold.

\cF is the set of formulas, $\cF \in \L{\omega+1}$, and $\models_\L{\alpha}$ is the satisfaction relation for \L{\alpha},
\[
        {\models_\L{\alpha}} (\phi,\vec{s}) \LEquiv \phi \in \cF \And \vec{s} \in \LL{\alpha}{<\omega} \And \L{\alpha} \models \phi[\vec{s}].
\]

Apart from the classic Condensation Lemma (see \ref{subsec:Constructibility}), we shall need the following familiar~result:\\
$(\ast)$~~\emph{For any limit $\lambda > \omega$, and $\beta < \lambda$, $\models_\L{\beta} \in \L{\lambda}$}. See \cite[§7.1]{Wes00}.

\begin{notation}
    Let $X \surj{\lambda} Y$ abbreviate $\Exists{f \In \L{\lambda}}(f : X \onto Y)$, where `$\onto$' stands for surjective map.
\end{notation}

Recall: Here, "$\mu$ is an \Lcard{\lambda}" means: "for no $\xi < \mu$, does $\xi \surj{\lambda} \mu$" (see~\ref{def:cardinality}).

\begin{lemma}
    Let $\lambda > \omega$ be limit. For\/ $0 < \alpha \le \gamma < \lambda$, 
    if\/ $\L{\beta} = \tHull{\gamma}(\alpha)$, then $\seq\alpha \surj{\lambda} \beta$.
\end{lemma}

\begin{proof}
Observe that $\L{\beta} = \Hull{\beta}(\alpha)$, and $\beta < \lambda$. In \L{\beta}, every $\xi < \beta$ is the unique solution of some formula $\phi(v, \vec{\*\eta})$, where $\vec{\eta} \in \seq\alpha$.
Thus, using $\models_\L{\beta} \in \L{\lambda}$, per $(\ast)$ above, one readily derives $\cF \X \seq\alpha \surj{\lambda} \beta$.
Using an injection $\cF \X \seq\alpha \to \seq\alpha$ in \L{\lambda}, one gets $\seq\alpha \surj{\lambda} \beta$.
\end{proof}

\begin{proposition}
\label{prop:Appendix-prop}
    Let $\lambda > \omega$ be a limit ordinal, and $\omega < \mu < \lambda$, an \Lcard{\lambda}.
    \begin{enumerate}
        \item	For $0 < \alpha < \mu \le \gamma < \lambda$, and\/ $\L{\beta} = \tHull{\gamma}(\alpha)$\,\up: $\beta < \mu$ 
                        \;\up(this is a \,$\downarrow$-Löwenheim-Skolem~property\up).
        \item 	$\mu$ is admissible.
    \end{enumerate}
\end{proposition}

\begin{proof}
We check (1) and (2) simultaneously, by induction on $\mu$.%
\smallbreak

(1)\enspace Set $\xbar\mu = \min_{\eta \le \mu}(\seq\eta \surj{\lambda} \mu)$. 
Note that, for $\eta < \xbar\mu$,  $(\eta \surj{\lambda} \xbar\mu \Implies \seq\eta \surj{\lambda} \seq{\xbar\mu} \surj{\lambda} \mu)$, 
it~follows that $\xbar\mu$ is an \Lcard{\lambda}, and clearly $\omega < \xbar\mu \le \mu$.

We claim that $\xbar\mu = \mu$.
If $\mu = \AlephL{1}{\lambda}$, then $\xbar\mu = \mu$.
Else, if $\xbar\mu < \mu$ then, by induction, $\xbar\mu$ is admissible, yielding an \L{\xbar\mu}\nb-definable map $\xbar\mu \onto \seq{\xbar\mu}$.
Whence $\xbar\mu \surj{\lambda} \seq{\xbar\mu} \surj{\lambda} \mu$, and thus $\xbar\mu \surj{\lambda} \mu$, contradicting "$\mu$ is an \Lcard{\lambda}".%

Now, given $0 < \alpha < \mu \le \gamma < \lambda$, and $\L{\beta} = \tHull{\gamma}(\alpha)$, the previous lemma yields $\seq{\alpha} \surj{\lambda} \beta$.
Hence, since $\alpha < \xbar\mu = \mu$, $\beta < \mu$.%
\smallbreak

(2)\enspace To show that $\mu$ is admissible, only $\Delta_0$ \Ax{Collection} needs checking.

Say $\L{\mu} \models \All{x \In \*a} \Exists{y} \phi(x, y, \vec{\*p})$, where $\phi$ is $\Delta_0$, and $a, \vec{p} \in \L{\mu}$.
Pick $\alpha < \mu$ with $a,\vec{p} \in \L{\alpha}$ and set $\L{\beta} = \tHull{\mu}(\alpha)$: $\L{\beta} \models \All{x \In \*a} \Exists{y} \phi(x, y, \vec{\*p})$.
Applying (1), $\beta < \mu$, thus $b =^\up{def} \L{\beta} \in \L{\mu}$.
By $\Delta_{0}$ absoluteness, $\L{\mu} \models \All{x \In \*a} \Exists{y \In \*b} \phi(x,y,\vec{\*p})$.
\end{proof}

\begin{namedthm*}{{\normalfont\ref{prop:Folklore}.} Proposition}
    For $\lambda$ limit, $\L{\lambda} \models "\*\mu > \omega \up{ is a successor \cardinal}" \Implies {\L{\mu} \models \ZFm}$.
\end{namedthm*}

\begin{proof}
Set $\pi=$ the \cardinal preceding $\mu$ in \L{\lambda}.
We argue that $\pi$ is the largest \cardinal in \L{\mu}.
Indeed, for $\pi \le \eta < \mu$, pick $\gamma < \lambda$ such that $\Exists{f \in \L{\gamma}}(f : \pi \onto \eta)$, and set $\L{\beta} = \tHull{\gamma}(\eta+1)$.
We get $\Exists{f \in \L{\beta}}(f : \pi \onto \eta)$ and, invoking \ref{prop:Appendix-prop}(1), $\beta < \mu$.
Hence $\L{\mu} \models \Exists{f}(f : \*\pi \onto \*\eta)$.
\smallbreak

Next: $\mu$ is regular in \L{\lambda}.
The usual \ZFC proof for the regularity of infinite successors goes through here:
for each nonzero $\eta < \mu$, using $<_\L{\mu}$, select $f_\eta \in \L{\mu}$, $f_\eta : \pi \onto \eta$, and note that the sequence $(f_\eta)_{0 < \eta < \mu}$ is in $\L{\mu+1} \subs \L{\lambda}$, etc.
\smallbreak

Finally, to show $\L{\mu} \models \ZFm$:
since by \ref{prop:Appendix-prop}(2) $\mu$ is admissible, using the standard definable bijection $\mu \to \L{\mu}$,
it suffices to verify \Ax{Replacement} for class-functions $\mu \to \mu$ in \L{\mu}.

Let therefore $F : \mu \to \mu$ be \L{\mu}-definable, from parameters $\vec{p}$. 
Given a set of ordinals $S \in \L{\mu}$, $S$~is~bounded in $\mu$.
By regularity of $\mu$ in \L{\lambda}, $F[S]$ is bounded as well.
Pick $\alpha < \mu$, with $F[S] \subs \alpha$ and $S, \vec{p} \in \L{\alpha}$\,:
$F[S]$ is definable over \L{\mu} from $S, \vec{p} \in \L{\alpha}$, and $\L{\alpha} \subs \Hull{\mu}(\alpha) \prec \L{\mu}$.
Set $\L{\beta} = \tHull{\mu}(\alpha)$, applying \ref{prop:Appendix-prop}(1), $\beta < \mu$, and thus $F[S] \in \L{\beta+1} \subs \L{\mu}$.
\end{proof}

\enlargethispage{8pt}

\bibliographystyle{ams}


\end{document}